\documentclass[11pt]{article}
\usepackage[margin=1in]{geometry}
\usepackage{amsthm,amsmath}
\usepackage{latexsym,amssymb,amsmath}
\usepackage{graphicx,float,color,fancybox,shapepar,setspace,hyperref}
\usepackage{subfigure}
\usepackage{pgf,tikz}
\usepackage[affil-it]{authblk}
\usepackage{indentfirst}
\usepackage{hyperref}
\usepackage[section]{placeins}
\usepackage{enumitem}
\usepackage{lineno}
\hypersetup
{
colorlinks=true,
linkcolor=blue,
filecolor=blue,
urlcolor=blue,
citecolor=cyan,
}
\usetikzlibrary{arrows}
\makeatletter
\renewenvironment{proof}[1][\proofname]{\par
  \normalfont \topsep6\p@\@plus6\p@\relax
  \trivlist
  \item[\hskip\labelsep
        \itshape
    #1\@addpunct{.}]\ignorespaces
}{%
  \hfill $\square$ % 手动添加白色方块
  \endtrivlist\@endpefalse
}
\makeatother

\newtheorem{definition}{Definition}
\newtheorem{theorem}{Theorem}
\newtheorem{lemma}{Lemma}
\newtheorem{proposition}{Proposition}
\newtheorem{corollary}{Corollary}

\newtheorem{claim}{Claim}

\newtheorem{problem}{Problem}

 \def\qed{\hfill\square}

\def\~{\sim}

\def\qed{ \hfill $\blacksquare$}

\begin{document}

\title{The maximum number of triangles in graphs without vertex disjoint friendship graphs}

\author[1]{Wanfang Chen\footnote{Email: \texttt{a372959313@gmail.com}}}

\author[2]{Jia-Bao Yang\footnote{Corresponding author. Email: \texttt{jbyang1215@163.com}}}

\author[3]{Leilei Zhang\footnote{Email: \texttt{mathdzhang@163.com}}}

\date{}

\affil[1]{\small School of Mathematical Sciences, University of Science and Technology of China, Hefei, 230026,  P.R. China}
\affil[2]{\small School of Mathematics, Nanjing University, Nanjing, 210093,  P.R. China }
\affil[3]{\small Faculty of Environment and Information Sciences, Yokohama National University, Yokohama 240-8501, Japan}

\maketitle

\begin{abstract}
Given graphs $H$ and $F$, the generalized Tur\'an number $\mathrm{ex}(n,H,F)$ is the maximum number of copies of $H$ among all $n$-vertex $F$-free graphs. The friendship graph $F_k$ consists of $k$ triangles sharing a common vertex. In this paper, we determine the value of $\mathrm{ex}(n,K_3,(t+1)F_k)$, where $K_3$ is a triangle, $t\geq 1$ is an integer, and $(t+1)F_k$ denotes a union of $(t+1)$ pairwise vertex-disjoint copies of $F_k$. Moreover, we characterize the extremal structure. Our result can be viewed as a generalization of the result of Zhu, Chen, Gerbner, Gy\H{o}ri, and Hama Karim, as well as of the remaining case left open by Wang, Ni, Liu, and Kang. In contrast to the extremal graphs of $F_k$, the extremal graphs of $(t+1)F_k$ undergo a fundamental change. This structure is also different from those of previous similar problems.
\end{abstract}


\section{Introduction}

For a graph $G$, we let $V(G)$ and $E(G)$ denote the vertex set and the edge set of $G$, respectively.
We say that a graph $G$ is {\it $F$-free} if it does not contain a copy of $F$ as a subgraph.
The {\it Tur\'an number} $\mathrm{ex}(n,F)$ of $F$ is the maximum number of edges in an $F$-free graph on $n$ vertices. 
The study of $\mathrm{ex}(n,F)$ is a central topic in extremal combinatorics, 
and one of the most famous results in this area is Tur\'an's \cite{turan1941} theorem, 
which states that the unique $n$-vertex $K_{r+1}$-free graph with the maximum number of edges is the balanced complete $r$-partite graph on $n$ vertices $T_r(n)$, 
where $K_{r+1}$ is the complete graph on $r+1$ vertices.

A {\it matching} in $G$ is a set of pairwise disjoint edges in $G$. 
The {\it matching number} $\nu(G)$ is the number of edges in a maximum matching of $G$.
Let $e(G)$ denote the number of edges in $G$. 
For $v\in V(G)$, let $N_G(v)$ denote the neighborhood of $v$ in $G$, and let $d_G(v)=|N_G(v)|$ denote its degree.
Let $\Delta(G)$ be the maximum degree of $G$. 
Define $$f(\nu,\Delta)=\max\{e(G):\nu(G)\leq \nu, \Delta(G)\leq \Delta\}.$$
In 1972, Abbott, Hanson, and Sauer \cite{AbbottHS1972} determined the value of $f(k-1,k-1)$.
That is
\begin{eqnarray*}
f(k-1,k-1)=
  \begin{cases}
  k(k-1),& k\text{ odd},\\[1mm]
  k\left(k-\frac{3}{2}\right),& k\text{ even}.
  \end{cases}
\end{eqnarray*}

Later Chv\'atal and Hanson \cite{ChvatalH1976} proved the following theorem.

\begin{theorem}[Chv\'atal and Hanson \cite{ChvatalH1976}]\label{thm:ChvatalH1976}
For all integers $\nu\geq 1$ and $\Delta\geq 1$,
\begin{equation*}
  f(\nu,\Delta)=\nu\Delta+
  \left\lfloor \frac{\Delta}{2} \right\rfloor 
  \left\lfloor \frac{\nu}{\left\lceil \Delta/2 \right\rceil} \right\rfloor.
\end{equation*}
\end{theorem}

The friendship graph $F_k$ consists of $k$ triangles sharing a common vertex $v$, called the center of $F_k$.
For $V\subseteq V(G)$, the subgraph of $G$ induced by $V$ is denoted by $G[V]$. 
A graph $G$ contains a copy of $F_k$ with center $v$ if and only if the induced subgraph $G[N_G(v)]$ contains a matching of size $k$.
In 1995, Erd\H{o}s, Füredi, Gould, and Gunderson \cite{ErdosFGG1995} determined the Tur\'an number of the friendship graph as follows.

\begin{theorem}[Erd\H{o}s, Füredi, Gould, and Gunderson \cite{ErdosFGG1995}]\label{the:ErdosFGG1995}
For every $k \geq 1$ and $n\geq 50k^2$, 
$$\mathrm{ex}(n,F_k)=\left\lfloor \frac{n^2}{4}\right\rfloor+f(k-1,k-1).$$
\end{theorem}

An {\em generalized friendship graph} consists of $k$ copies of $K_r$ sharing a common vertex.
In particular, $F_k^3 = F_k$. Chen, Gould, Pfender, and Wei \cite{ChenGPW2003} extended Erd\H{o}s, Füredi, Gould, and Gunderson's results to the generalized case.

\begin{theorem}[Chen, Gould, Pfender, and Wei \cite{ChenGPW2003}]\label{the:ChenGPW2003}
For every $k \geq 1$, $r \geq 2$ and $n \geq 16k^3r^8$, 
$$\mathrm{ex}(n,F_k^r)=e(T_{r-1}(n))+f(k-1,k-1).$$
\end{theorem}

After that, Hou, Li, and Zeng \cite{HouLZ2024}, and Chen, Lei, and Li \cite{ChenLL2025} considered the suspension of edge-critical graphs.

A natural generalization of the Tur\'an number is to count copies of a fixed graph $H$ instead of edges. 
For graphs $H$ and $G$, let $N(H,G)$ denote the number of copies of $H$ in $G$. 
The {\it generalized Tur\'an number} $\mathrm{ex}(n,H,F)$ is defined as
$$\mathrm{ex}(n,H,F):=\max\{N(H,G):|V(G)|=n,\ G \text{ is $F$-free}\}.$$
This problem was considered by Erd\H{o}s \cite{Erdos1962}, who determined $\mathrm{ex}(n, K_s, K_t)$ for all $t > s \geq 3$. 
Since then, it has attracted considerable attention; see, for example, \cite{Bollobas2008,Chase2020,Chenya2024,Furedi2015,Grzesik2012,Hatami2013,Luo2017,Yang2024}.

A natural problem is to determine the maximum number of copies of $K_3$ in $F_k$-free graphs.
In 2016, Alon and Shikhelman \cite{Alon2016} provided an upper bound on the number of triangles in $F_k$-free graphs.

\begin{theorem}[Alon and Shikhelman \cite{Alon2016}]\label{thm:Alon2016}
For every $k \geq 2$,
$$\mathrm{ex}(n,K_3,F_k)<(9k-15)(k+1)n.$$
\end{theorem}

Zhu, Chen, Gerbner, Gy\H{o}ri, and Hama Karim ~\cite{ZhuChen2023} determined $\mathrm{ex}(n,K_3,F_k)$ exactly and characterized the extremal graphs.

\begin{theorem}[Zhu, Chen, Gerbner, Gy\H{o}ri, and Hama Karim~\cite{ZhuChen2023}]\label{thm:ZhuChen2023}
Let $k \geq 3$ and $n \geq 4k^3$. 
Then
$$
\mathrm{ex}(n,K_3,F_k)=
\begin{cases}
(n-2k)k(k-1)+2\binom{k}{3}, & k \text{ odd},\\[1mm]
(n-2k+1)k\left(k-\frac{3}{2}\right)+2\binom{k-1}{3}+\left(\frac{k}{2}-1\right)^2, & k \text{ even}.
\end{cases}
$$
Moreover, the extremal graphs are characterized.
\end{theorem}

For two graphs $G_1$ and $G_2$, let $G_1\cup G_2$ denote the graph with vertex set $V(G_1)\cup V(G_2)$ and edge set $E(G_1)\cup E(G_2)$.
For a positive integer $t$, we use $tG$ to denote $t$ pairwise vertex-disjoint copies of $G$.
If $G_1$ and $G_2$ are vertex-disjoint graphs, then $G_1\vee G_2$ denotes their \emph{join}, obtained from $G_1\cup G_2$ by adding all edges between $V(G_1)$ and $V(G_2)$.

Recently, Wang, Ni, Liu, and Kang~\cite{Wang2024} studied the generalized Tur\'an number of $t$ pairwise vertex-disjoint generalized friendship graphs for $r\geq 4$.
%\begin{theorem}[Wang et al. \cite{Wang2024}]\label{thm:Wang2024}
%For integers $t\geq 1$, $k_1\geq k_2\geq \cdots\geq k_t\geq 1$, $r>k_t+1$ and $r\geq m+1\geq 4$,
%$$\mathrm{ex}\left(n,K_m,\bigcup_{i=1}^t F_{k_i}^r\right)=$$
%\end{theorem}
Motivated by this and Theorem~\ref{thm:ZhuChen2023}, we investigate the generalized Tur\'an number of $(t+1)$ vertex-disjoint friendship graphs.

To state our main results, we introduce some notation.

\begin{definition}[The graph family $\mathcal{P}_k$]
Let $\mathcal{P}_k$ be the family of all graphs $P$ with no isolated vertices such that
\begin{equation*}\label{eq:mathcal-p-family}
  \nu(P)\le k-1,\qquad \Delta(P)\le k-1,\qquad e(P)=f(k-1,k-1).
\end{equation*}
\end{definition}

Inspired by Erd\H{o}s, Füredi, Gould, and Gunderson \cite{ErdosFGG1995}, we introduce the following definition.

\begin{definition}[$k$-admissible]
Let $P, Q \in \mathcal{P}_k$, $A=V(P)$ and $B=V(Q)$, and let $R$ be a bipartite graph with parts $A$ and $B$.  
We call $(P,Q,R)$ \emph{$k$-admissible} if
\begin{equation}\label{eq:adm-A}
  d_P(a)+\nu\bigl(Q[N_R(a)]\bigr)\le k-1
  \qquad\text{for every }a\in A,
\end{equation}
and
\begin{equation}\label{eq:adm-B}
  d_Q(b)+\nu\bigl(P[N_R(b)]\bigr)\le k-1
  \qquad\text{for every }b\in B.
\end{equation}
\end{definition}

Given two vertex sets,  we write $A \times B=\{ab:a\in A, b\in B\}$. 
%We write $A \times B$ for the complete bipartite graph with parts $A$ and $B$, where $A$ and $B$ are vertex sets.
To characterize the extremal graphs in our main result, we now construct a graph $H_n(P,Q,R)$ from a $k$-admissible triple $(P,Q,R)$.

\begin{definition}
Given a $k$-admissible triple $(P,Q,R)$, we construct $H_n(P,Q,R)$ as follows.  
Take a balanced partition $X\cup Y$ of an $n$-vertex set.  
Place a copy of $P$ on a subset $A\subseteq X$ and a copy of $Q$ on a subset $B\subseteq Y$.  
All vertices in $X\setminus A$ and $Y\setminus B$ are isolated within their own parts.  
The edges between $X$ and $Y$ consist of $\bigl((X\times Y)\setminus(A\times B)\bigr)\cup E(R)$.
\end{definition}

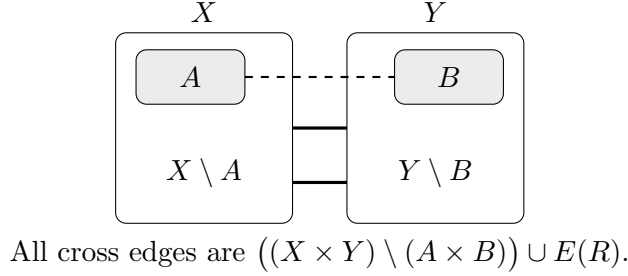
\begin{figure}[ht]
\centering
\begin{tikzpicture}[scale=0.9]
  \draw[rounded corners] (-3,-1.4) rectangle (-0.4,1.4);
  \draw[rounded corners] (0.4,-1.4) rectangle (3,1.4);
  \node at (-1.7,1.7) {$X$};
  \node at (1.7,1.7) {$Y$};
  \draw[rounded corners,fill=gray!15] (-2.7,0.35) rectangle (-1.1,1.15);
  \draw[rounded corners,fill=gray!15] (1.1,0.35) rectangle (2.7,1.15);
  \node at (-1.9,0.75) {$A$};
  \node at (1.9,0.75) {$B$};
  \node at (-1.7,-0.65) {$X\setminus A$};
  \node at (1.7,-0.65) {$Y\setminus B$};
  \draw[very thick] (-0.4,-0.8) -- (0.4,-0.8);
  \draw[very thick] (-0.4,0.0) -- (0.4,0.0);
  \draw[dashed,thick] (-1.1,0.75) -- (1.1,0.75);
  \node at (0,-1.85) {All cross edges are $\bigl((X\times Y)\setminus(A\times B)\bigr)\cup E(R)$.};
\end{tikzpicture}
\caption{Construct $H_n(P,Q,R)$ from a $k$-admissible triple $(P,Q,R)$.}\label{fig:construction}
\end{figure}

For convenience, we let
\begin{equation}\label{eq:tau}
  \tau_R(P,Q)=
  \sum_{aa'\in E(P)} |N_R(a)\cap N_R(a')|
  +
  \sum_{bb'\in E(Q)} |N_R(b)\cap N_R(b')|
\end{equation}
and
\begin{eqnarray}\label{eq:Phi}
  \Phi(P,Q,R,t)&=&( 2f(k-1,k-1)-|A||B|+e(R))t-f(k-1,k-1)(|A|+|B|)\nonumber\\
  &&\quad +N(K_3,P)+N(K_3,Q)+\tau_R(P,Q).
\end{eqnarray}
Since $\mathcal{P}_k$ is finite and there are only finitely many choices for $R\subseteq A\times B$, the number
\begin{equation*}
  c_k^*(t)=\max\{\Phi(P,Q,R,t):(P,Q,R)\text{ is }k\text{-admissible}\}
\end{equation*}
is well defined.

We establish the following theorem, which gives the result for $\mathrm{ex}(n,K_3,(t+1)F_k)$.
\begin{theorem}\label{thm:main}
Let $k\geq 3$ and $t\geq 1$ be integers.
For sufficiently large $n$,
\begin{equation}\label{eq:main}
  \mathrm{ex}(n,K_3,(t+1)F_k)=\binom{t}{3}+(n-t)\binom{t}{2}+t\left\lfloor \frac{(n-t)^2}{4} \right\rfloor+f(k-1,k-1)(n-t)+c^*_k(t).
\end{equation}
Moreover, the extremal graphs are of the form $K_t \vee H_{n-t}(P,Q,R)$, where $(P,Q,R)$ is a $k$-admissible triple satisfying $\Phi(P,Q,R,t)=c_k^*(t)$.
\end{theorem}

It is worth noting that, different from the extremal graphs of $F_k$, 
the extremal graphs of $(t+1)F_k$ undergo a fundamental change.
Moreover, this structure is different from those of previous similar problems.

Before proving Theorem~\ref{thm:main}, we first establish the following theorem.
%In particular, Lemma~\ref{lem:stability-for-T} (see Section~\ref{Proof-main-Theorem-using-reduce-theorem}) establishes a connection between Theorem~\ref{thm:main} and the following theorem.
Let $g(k,t,n)$ be defined as follows:
$$g(k,t,n)=\max\{t e(H)+N(K_3,H): |V(H)|=n,\ H\text{ is }F_k\text{-free}\}.$$

\begin{theorem}\label{thm:main-reduce}
Let $k\geq 3$ and $t\geq 1$ be integers.
For sufficiently large $n$,
\begin{equation}\label{eq:exact-M}
g(k,t,n)=t\left\lfloor \frac{n^2}{4} \right\rfloor+f(k-1,k-1)n+c_k^*(t).
\end{equation}
Moreover, a graph $H$ attains $g(k,t,n)$ if and only if $H=H_n(P,Q,R)$ for some $k$-admissible triple $(P,Q,R)$ with $\Phi(P,Q,R,t)=c_k^*(t)$.
\end{theorem}

The organization of this paper is as follows. 
We first introduce the necessary preliminaries in the next section. 
In Section~\ref{Proof-main-Theorem-using-reduce-theorem}, we prove Theorem~\ref{thm:main} using Theorem~\ref{thm:main-reduce}. 
In Section~\ref{Proof-reduce-theorem}, we establish Theorem~\ref{thm:main-reduce}. 
We conclude in Section~\ref{concluding-remarks} with a further problem.

\section{Preliminaries}
For $S\subseteq V(G)$, we write $G\setminus S$ for the induced subgraph $G[V(G)\setminus S]$. 
For subsets $X,Y\subseteq V(G)$, let $E_G(X,Y)$ denote the set of edges in $G$ with one endpoint in $X$ and the other in $Y$, and let $e_G(X,Y)=|E_G(X,Y)|$.
If $xy\notin E(G)$, we use $G+xy$ to denote the graph obtained from $G$ by adding the edge $xy$.

We shall use the following theorem of Erd\H{o}s and Gallai \cite{Erdos1959}, 
which determines the maximum number of edges in an $n$-vertex graph with matching number at most $k-1$.

\begin{lemma}[Erd\H{o}s and Gallai \cite{Erdos1959}] \label{lem:Erdos1959}
Let $G$ be a graph on $n$ vertices. If $\nu(G)\leq k-1$, then
\begin{equation*}
e(G)\leq \max\left\{\binom{2k-1}{2},\binom{k-1}{2}+(k-1)(n-k+1)\right\}.
\end{equation*}
\end{lemma}

\begin{lemma}\label{lem:construction-is-free}
If $(P,Q,R)$ is $k$-admissible, then $H_n(P,Q,R)$ is $F_k$-free.
\end{lemma}

\begin{proof}
  Let $G=H_n(P,Q,R)$. It suffices to show that
$\nu(G[N_G(v)])\le k-1$ for every $v\in V(G)$.
Indeed, if $G$ contains a copy of $F_k$ centered at $v$, then $G[N_G(v)]$ contains a matching of size $k$.

If $x\in X\setminus A$, then $N_G(x)=Y$. Moreover, the induced subgraph $G[Y]$ is obtained from $Q$ by adding isolated vertices. Hence
$$
\nu(G[N_G(x)])=\nu(G[Y])=\nu(Q)\le k-1.
$$
The case $y\in Y\setminus B$ is symmetric.

Let $a\in A$.
Note that every edge of $G[N_G(a)]$ either has an endpoint in $N_P(a)$ or lies in $Q[N_R(a)]$.  
Hence every matching in $G[N_G(a)]$ has size at most
$$
d_P(a)+\nu(Q[N_R(a)])\le k-1
$$
by the $k$-admissibility of $(P,Q,R)$.  
The argument for vertices $b\in B$ is identical.  

Therefore no vertex is the center of a copy of $F_k$, and so $G$ is $F_k$-free.
\end{proof}

The following lemma gives the lower bound of Theorem~\ref{thm:main-reduce}.

\begin{lemma}\label{lem:lower-bound-main1}
For every $k$-admissible triple $(P,Q,R)$ and sufficiently large $n$,
$$te(H_n(P,Q,R))+N(K_3,H_n(P,Q,R))=t\left\lfloor{\frac{n^2}{4}}\right\rfloor+f(k-1,k-1)n+\Phi(P,Q,R,t).$$
Consequently, we have
$$g(k,t,n)\geq t\left\lfloor{\frac{n^2}{4}}\right\rfloor+f(k-1,k-1)n+c_k^*(t).$$
\end{lemma}

\begin{proof}
By the definition of $H_n(P,Q,R)$, we have $|X||Y|=\left\lfloor n^2/4\right\rfloor$.  
Since $e(P)=e(Q)=f(k-1,k-1)$, we obtain
\begin{align*}
e(H_n(P,Q,R))
&=|X||Y|-|A||B|+e(R)+2f(k-1,k-1)\\
&=\left\lfloor \frac{n^2}{4}\right\rfloor-|A||B|+e(R)+2f(k-1,k-1).
\end{align*}

Every edge of $P$ forms a triangle with each vertex of $Y\setminus B$ and with each common neighbor in $R$ of its two endpoints in $B$. 
The analogous statement holds for $Q$.  
Adding the internal triangles of $P$ and $Q$, we get
\begin{eqnarray*}
&&N(K_3,H_n(P,Q,R))\\
&=&f(k-1,k-1)(|Y|-|B|)+f(k-1,k-1)(|X|-|A|)+N(K_3,P)+N(K_3,Q)\\
&&+\sum_{aa'\in E(P)} |N_R(a)\cap N_R(a')|+\sum_{bb'\in E(Q)} |N_R(b)\cap N_R(b')|\\
&=&f(k-1,k-1)(n-|A|-|B|)+N(K_3,P)+N(K_3,Q)+\tau_R(P,Q).
\end{eqnarray*}
The second equality is based on \eqref{eq:tau}.
Recall that 
\begin{eqnarray*}
  \Phi(P,Q,R,t)&=&( 2f(k-1,k-1)-|A||B|+e(R))t-f(k-1,k-1)(|A|+|B|)\nonumber\\
  &&\quad +N(K_3,P)+N(K_3,Q)+\tau_R(P,Q).
\end{eqnarray*}
Hence, the sum of the two displayed formulae is exactly the asserted identity.
\end{proof}

\section{Proof of Theorem \ref{thm:main} using Theorem \ref{thm:main-reduce}}\label{Proof-main-Theorem-using-reduce-theorem}

The main strategy for proving Theorem~\ref{thm:main} is as follows.
We first apply Lemma~\ref{lem:construction-is-free} to establish the lower bound in Theorem~\ref{thm:main}.
Then, we identify a special vertex set $T$ of size $t$ such that $G\setminus T$ is $F_k$-free.
Therefore, we can derive the upper bound using Theorem~\ref{thm:main-reduce}.
Moreover, by combining the equality of the upper and lower bounds with Theorem~\ref{thm:main-reduce}, 
we characterize the extremal graphs in Theorem~\ref{thm:main}.

Next, we prove a stability lemma for $(t+1)F_k$-free graphs satisfying certain properties.

\begin{lemma}\label{lem:stability-for-T}
Let $k\geq 1$ and $t\geq 1$ be integers. 
Let $G$ be an $n$-vertex $(t+1)F_k$-free graph with 
$$ N(K_3,G)\ge \left(\frac{t-1}{4}+\delta\right)n^2,$$
where $0<\delta<1/4$ is a constant.
For sufficiently large $n$, there exists a vertex set $T\subseteq V(G)$ with $|T|=t$ such that $G\setminus T$ is $F_k$-free.
\end{lemma}

\begin{proof}
Since $N(K_3,G)=\Omega(n^2)$, Theorem~\ref{thm:ZhuChen2023} implies that $G$ contains a copy of $F_k$ for sufficiently large $n$.
Let $\ell$ be the maximum number of pairwise vertex-disjoint copies of $F_k$ in $G$, 
and let $S$ be its vertex set.
Since $G$ is $(t+1)F_k$-free, we have $\ell\leq t$.
Moreover, $|S|=(2k+1)\ell$ and $G\setminus S$ is $F_k$-free.
Applying Theorem~\ref{thm:ZhuChen2023} again, we have
\begin{equation}\label{eq:pf:stability-for-T:eq1}
  N(K_3,G\setminus S)=O(n).
\end{equation}
The number of triangles intersecting $S$ in at least two vertices is $O(n)$. 
Therefore, by a straightforward counting argument,
\begin{equation}\label{eq:link-sum}
  N(K_3,G)=\sum_{v\in S} e\bigl(G[N_G(v)\setminus S]\bigr)+O(n).
\end{equation}

Define the set of vertices with high-edge in their neighborhood graphs by
$$T=\left\{v\in S: e\bigl(G[N_G(v)\setminus S]\bigr) > \frac{\delta n^2}{8k(t+1)}\right\}.$$
We now prove that $|T|=t$.

\begin{claim}\label{claim:pf:stability-for-T:claim1}
$|T|\leq t$.
\end{claim}

\begin{proof}
Assume that $|T|\geq t+1$, and choose distinct vertices $v_i\in T$ for $i\in [t+1]$. 
Since the graph $G[N_G(v_1)\setminus S]$ has more than 
$\delta n^2/(8k(t+1))>kn$ edges for sufficiently large $n$, Lemma~\ref{lem:Erdos1959} implies that it contains a matching $M_1$ of size $k$. 

We greedily construct matchings $M_1,\ldots,M_{t+1}$ in $G\setminus S$ as follows, where 
$M_i \subseteq E\bigl(G[N_G(v_i)\setminus S]\bigr)$ is a matching of size $k$, and the matchings $M_1,\ldots,M_{t+1}$ are pairwise vertex-disjoint.
Assume that we have already found $M_1,\ldots,M_j$ with these properties, where $1\leq j\leq t$.
Delete the $2kj$ vertices covered by $M_1\cup\cdots\cup M_j$ from $G[N_G(v_{j+1})\setminus S]$. 
This deletes at most $2kjn$ edges, so the remaining graph still has more than
$$
\frac{\delta n^2}{8k(t+1)}-2kjn>kn
$$
edges for sufficiently large $n$.
Thus, Lemma~\ref{lem:Erdos1959} implies that the remaining graph contains a matching $M_{j+1}$ of size $k$.

For each $i\in[t+1]$, the vertex $v_i$ together with the $k$ edges of $M_i$ forms a copy of $F_k$ centered at $v_i$.
Since the vertices $v_1,\ldots,v_{t+1}$ are distinct and the matchings $M_1,\ldots,M_{t+1}$ are pairwise vertex-disjoint in $G\setminus S$, these copies of $F_k$ are pairwise vertex-disjoint.
This contradicts the assumption that $G$ is $(t+1)F_k$-free.
\end{proof}

\begin{claim}\label{claim:pf:stability-for-T:claim2}
$|T|\geq t$.
\end{claim}
\begin{proof}
Assume that $|T|\leq t-1$. 
Since $G\setminus S$ is $F_k$-free, it follows from Theorem~\ref{the:ErdosFGG1995} that 
$$e\bigl(G[N_G(v)\setminus S]\bigr)\leq \frac{n^2}{4}+O(n)$$
for every $v\in T$.
By the definition of $T$, for every $v\in S\setminus T$, we have
$$e\bigl(G[N_G(v)\setminus S]\bigr)\leq \frac{\delta n^2}{8k(t+1)}.$$
Since $|S|\leq (2k+1)(t+1)$ and $|T|\leq t-1$, we deduce from \eqref{eq:link-sum} that 
\begin{align*}
N(K_3,G)
&=\sum_{v\in S} e\bigl(G[N_G(v)\setminus S]\bigr)+O(n)\\
&=\sum_{v\in T} e\bigl(G[N_G(v)\setminus S]\bigr)
  +\sum_{v\in S\setminus T} e\bigl(G[N_G(v)\setminus S]\bigr)+O(n)\\
&\leq \frac{|T|}{4}n^2+\frac{|S|-|T|}{8k(t+1)}\delta n^2+O(n)\\
%&\leq \frac{|T|}{4}n^2+\frac{|S|}{8k(t+1)}\delta n^2+O(n)\\
&\leq \frac{t-1}{4}n^2+\frac{2k+1}{8k}\delta n^2+O(n)\\
&< \left(\frac{t-1}{4}+\delta\right)n^2
\end{align*}
for sufficiently large $n$. 
This contradicts the hypothesis that
$$N(K_3,G)\geq \left(\frac{t-1}{4}+\delta\right)n^2.$$
Thus $|T|\geq t$.
\end{proof}

Hence, $|T|=t$.
Without loss of generality, assume that $T=\{u_1,\ldots,u_t\}$.
Then, for every $u_i\in T$, we have
\begin{equation}\label{eq:u-quadratic-link}
  e\bigl(G[N_G(u_i)\setminus S]\bigr)=\Omega(n^2).
\end{equation}

Now we show that $G\setminus T$ is $F_k$-free. 
Suppose not, and let $\Lambda_0$ be a copy of $F_k$ in $G\setminus T$. 
For each $i\in[t]$, we greedily choose a matching $M_i$ of size $k$ in $G[N_G(u_i)\setminus S]$ so that the matchings $M_1,\ldots,M_t$ are pairwise vertex-disjoint and are also disjoint from $\Lambda_0$.
Indeed, at step $i$, deleting $V(\Lambda_0)\cup S$ together with the vertices covered by the previously chosen matchings removes only $O(n)$ edges from $G[N_G(u_i)\setminus S]$. 
By \eqref{eq:u-quadratic-link}, the remaining graph still has $\Omega(n^2)$ edges, and hence contains a matching of size $k$ by Lemma~\ref{lem:Erdos1959}. 

For each $i\in[t]$, the matching $M_i$, together with the center $u_i$, forms a copy of $F_k$. 
These $t$ copies are pairwise vertex-disjoint and are disjoint from $\Lambda_0$. 
Thus $G$ contains $t+1$ pairwise vertex-disjoint copies of $F_k$, contradicting the assumption that $G$ is $(t+1)F_k$-free.
Therefore, $G\setminus T$ is $F_k$-free.
\end{proof}

{\bf Proof of Theorem~\ref{thm:main} using Theorem~\ref{thm:main-reduce}.}
Let $(P,Q,R)$ be a $k$-admissible triple with $\Phi(P,Q,R,t)=c_k^*(t)$.
Then $H_{n-t}(P,Q,R)$ is $F_k$-free by Lemma~\ref{lem:construction-is-free}.
Lemma~\ref{lem:lower-bound-main1} shows that
\begin{align*}
&t e(H_{n-t}(P,Q,R))+N(K_3,H_{n-t}(P,Q,R))\\
&=t\left\lfloor \frac{(n-t)^2}{4} \right\rfloor+f(k-1,k-1)(n-t)+\Phi(P,Q,R,t) \\
&=t\left\lfloor \frac{(n-t)^2}{4} \right\rfloor+f(k-1,k-1)(n-t)+c_k^*(t).
\end{align*}
Hence $K_t \vee H_{n-t}(P,Q,R)$ is $(t+1)F_k$-free. 
%since the $t$ special vertices can each belong to at most one copy.
Indeed, after deleting the $t$ vertices of $K_t$, the remaining graph is $F_k$-free; therefore every copy of $F_k$ in $K_t \vee H_{n-t}(P,Q,R)$ must contain at least one vertex of $K_t$, 
and so there can be at most $t$ pairwise vertex-disjoint copies of $F_k$.
Moreover,
\begin{align*}
&N(K_3,K_t\vee H_{n-t}(P,Q,R))\\
&=\binom{t}{3}+(n-t)\binom{t}{2}
  +t e(H_{n-t}(P,Q,R))+N(K_3,H_{n-t}(P,Q,R))\\ 
&=\binom{t}{3}+(n-t)\binom{t}{2}
  +t\left\lfloor \frac{(n-t)^2}{4} \right\rfloor
  +f(k-1,k-1)(n-t)+c_k^*(t).
\end{align*}
Thus
\begin{equation}\label{eq:pf:main-lower}
\mathrm{ex}(n,K_3,(t+1)F_k)
\geq \binom{t}{3}+(n-t)\binom{t}{2}
+t\left\lfloor \frac{(n-t)^2}{4} \right\rfloor
+f(k-1,k-1)(n-t)+c_k^*(t).
\end{equation}

Conversely, let $G$ be an extremal $n$-vertex $(t+1)F_k$-free graph. 
By the lower bound above, $G$ has at least
$$t\left\lfloor \frac{(n-t)^2}{4} \right\rfloor+O(n)$$
triangles. 
In particular, for sufficiently large $n$
$$N(K_3,G)\ge \left(\frac{t-1}{4}+\frac{1}{5}\right)n^2.$$
Applying Lemma~\ref{lem:stability-for-T} with $\delta=1/5$, there exists a vertex set $T\subseteq V(G)$ with $|T|=t$ such that $G\setminus T$ is $F_k$-free.
Then
\begin{align}\label{eq:pf:main-upper}
N(K_3,G)
&\leq \binom{|T|}{3}+(n-|T|)\binom{|T|}{2}+ |T| e\bigl(G\setminus T\bigr)+N(K_3,G\setminus T)\nonumber\\
&\leq \binom{t}{3}+(n-t)\binom{t}{2}+t e\bigl(G\setminus T\bigr)+N(K_3,G\setminus T)\nonumber\\
&\leq \binom{t}{3}+(n-t)\binom{t}{2}+g(k,t,n-t)\nonumber\\
&\leq \binom{t}{3}+(n-t)\binom{t}{2}+t\left\lfloor \frac{(n-t)^2}{4} \right\rfloor+f(k-1,k-1)(n-t)+c_k^*(t).
\end{align}
The last inequality is due to Theorem \ref{thm:main-reduce}.
Together with \eqref{eq:pf:main-lower}, this proves \eqref{eq:main}.

Moreover, all inequalities in \eqref{eq:pf:main-upper} must be equalities. 
It follows from Theorem~\ref{thm:main-reduce} that
$G\setminus T=H_{n-t}(P',Q',R')$ for some $k$-admissible triple $(P',Q',R')$ with
$\Phi(P',Q',R',t)=c_k^*(t)$.
The equality in the first and second lines of \eqref{eq:pf:main-upper} further implies that
$G=K_t\vee (G\setminus T)$.
Hence
$$G=K_t\vee H_{n-t}(P',Q',R').$$
This completes the proof. 
\qed

\section{Proof of Theorem \ref{thm:main-reduce}}\label{Proof-reduce-theorem}
This section is organized into two parts. 
In the first part, we give some key lemmas; in the second, we prove Theorem~\ref{thm:main-reduce}.

\subsection{Key lemmas}
In extremal graph theory, graphs for which a parameter is close to its maximum often exhibit a similar structure.
This phenomenon is known as stability.
A classical theorem is given by Erd\H{o}s-Simonovits \cite{Erdos1966,Erdos1967,Simonovits1974}.

\begin{theorem}[Erd\H{o}s-Simonovits stability theorem \cite{Erdos1966,Erdos1967,Simonovits1974}]\label{the:Erdos-Simonovits-stability-theorem}
For every fixed graph $F$ with $\chi(F)=3$ and every $\varepsilon>0$, there exists a constant $\delta=\delta(F,\varepsilon)>0$ such that, for all sufficiently large $n$, the following holds.  
If $G$ is an $n$-vertex $F$-free graph with 
$$e(G)\ge \left(\frac{1}{4}-\delta\right)n^2,$$
then there is a partition $V(G)=U\cup W$ such that
$$e(G[U])+e(G[W])\le \varepsilon n^2.$$  
\end{theorem}

We apply this theorem only with $F=F_k$ to obtain the following stability result for friendship graphs.

\begin{lemma}[Stability result for friendship graphs]\label{lem:friendship-graphs-stability}
For every fixed $k\geq 3$ and every fixed $\alpha\geq 0$, 
there exists a constant $\beta=\beta(k,\alpha)$ such that, 
for sufficiently large $n$, the following holds.  
If $G$ is an $n$-vertex $F_k$-free graph with 
\begin{equation}\label{eq:friendship-graphs-stability-hypothesis}
  e(G)\ge \left\lfloor \frac{n^2}{4}\right\rfloor-\alpha,
\end{equation}
then there is a partition $V(G)=X\cup Y$ such that
\begin{equation}\label{eq:friendship-graphs-stability-upper-bounded}
  \bigl||X|-|Y|\bigr|\le \beta,
  \qquad
  e(G[X])+e(G[Y])\le \beta,
  \qquad
  |(X\times Y)\setminus E(G)|\le \beta.
\end{equation}
\end{lemma}

\begin{proof}
Choose a constant $\varepsilon>0$ such that
\begin{equation}\label{eq:pf:friendship-graphs-stability-eps}
  0<\varepsilon<\frac{1}{100}
  \qquad\text{and}\qquad
  48k^2\sqrt{2\varepsilon}\le \frac{1}{4}.
\end{equation}
For this $\varepsilon$, let $\delta=\delta(F_k,\varepsilon)$ be the constant given by Theorem~\ref{the:Erdos-Simonovits-stability-theorem}.  
Choose $n$ sufficiently large so that Theorem~\ref{the:Erdos-Simonovits-stability-theorem} applies with this choice of $\varepsilon$ and $\delta$, and
\begin{equation}\label{eq:pf:friendship-graphs-stability-n0}
  \left\lfloor \frac{n^2}{4} \right\rfloor-\alpha\ge \left(\frac{1}{4}-\delta\right)n^2,
  \qquad
  \alpha+1\le \varepsilon n^2,
  \qquad \text{and} \qquad 
  \frac{n^2}{36}-1>2\varepsilon n^2.
\end{equation}

Take a partition $V(G)=X\cup Y$ that maximizes $e_G(X,Y)$. 
For convenience, let
$$  s=e(G[X])+e(G[Y]),
  \qquad
  r=|(X\times Y)\setminus E(G)|,
  \qquad
q=\left\lfloor{\frac{n^2}{4}}\right\rfloor-|X||Y|.$$
Then $q\geq 0$, and the exact edge count is
$$e(G)=|X||Y|-r+s=\left\lfloor \frac{n^2}{4}\right\rfloor-q-r+s.$$
Combining this with \eqref{eq:friendship-graphs-stability-hypothesis} yields
\begin{equation}\label{eq:r-q-simple}
  q+r\leq s+\alpha.
\end{equation}

For $x\in X$ and $y\in Y$, let $d_{G[Y]}(x)=|N_{G}(x)\cap Y|$ and $d_{G[X]}(y)=|N_{G}(y)\cap X|$.
\begin{claim}\label{claim:pf:friendship-graphs-stabilitymaxcut-degree}
For every $x\in X$ and every $y\in Y$,
\[
d_{G[X]}(x)\le d_{G[Y]}(x)
\qquad\text{and}\qquad
d_{G[Y]}(y)\le d_{G[X]}(y).
\]
\end{claim}

\begin{proof}
Otherwise, suppose that there exists $x\in X$ such that $d_{G[X]}(x)>d_{G[Y]}(x)$.
Then moving $x$ from $X$ to $Y$ would increase the number of crossing edges by
$d_{G[X]}(x)-d_{G[Y]}(x)$, contradicting the maximality of $e_G(X,Y)$.
Similarly, we have $d_{G[Y]}(y)\le d_{G[X]}(y)$ for every $y\in Y$.
\end{proof}

We first get weak bounds for $r$ and $s$.
\begin{claim}\label{claim:pf:friendship-graphs-stability-bound-of-rs}
$r\le 2\varepsilon n^2$ and $q\le 2\varepsilon n^2$.
\end{claim}
\begin{proof}
By \eqref{eq:friendship-graphs-stability-hypothesis} and \eqref{eq:pf:friendship-graphs-stability-n0}, 
the Erd\H{o}s-Simonovits stability theorem gives a partition $U\cup W$ such that
$$e(G[U])+e(G[W])\le \varepsilon n^2.$$
For every bipartition $A\cup (V(G)\setminus A)$, we have
$$e(G)=e_G(A,V(G)\setminus A)+e(G[A])+e(G[V(G)\setminus A]).$$
Since the partition $X\cup Y$ maximizes $e_G(X,Y)$, the value $e(G[X])+e(G[Y])$ is minimized among all bipartitions of $V(G)$.
Therefore,
$$s=e(G[X])+e(G[Y])\le \varepsilon n^2.$$
Together with \eqref{eq:r-q-simple}, this gives
$$q+r\le s+\alpha\le \varepsilon n^2+\alpha.$$
Since $\alpha$ is fixed, for sufficiently large $n$ we have $\alpha\le \varepsilon n^2$, and hence
$$q+r\le 2\varepsilon n^2.$$
In particular, $r\le 2\varepsilon n^2$ and $q\le 2\varepsilon n^2$.
\end{proof}

\begin{claim}\label{claim:pf:friendship-graphs-stability-XYsize}
$\min\{|X|,|Y|\}\ge n/3$.
\end{claim}

\begin{proof}
Suppose, to the contrary, that $\min\{|X|,|Y|\}<n/3$.  
Without loss of generality, assume that $|X|<n/3$. 
Then $|X||Y|=|X|(n-|X|)<2n^2/9$, and hence
$$q=\left\lfloor \frac{n^2}{4}\right\rfloor-|X||Y|
    >\frac{n^2}{4}-1-\frac{2n^2}{9}
    =\frac{n^2}{36}-1>2\varepsilon n^2,$$
where the last inequality follows from \eqref{eq:pf:friendship-graphs-stability-n0}.
This leads to a contradiction with Claim \ref{claim:pf:friendship-graphs-stability-bound-of-rs}.
\end{proof}

The following claim bounds the maximum degrees inside the parts.

\begin{claim}\label{claim:pf:friendship-graphs-stability-max-degree}
$\Delta(G[X])\leq 4k\sqrt{r+1}$ and $\Delta(G[Y])\leq 4k\sqrt{r+1}$.
\end{claim}

\begin{proof}
Since $X$ and $Y$ are symmetric, we only need to prove that $\Delta(G[X])\leq 4k\sqrt{r+1}$.
It is enough to show that, for any fixed vertex $v\in X$, we have $d_{G[X]}(v)\leq 4k\sqrt{r+1}$.
If $d_{G[X]}(v)<2k$, then $d_{G[X]}(v)\leq 2k\leq 4k\sqrt{r+1}$, and we are done.
Thus, we may assume that $d_{G[X]}(v)\geq 2k$.  
Choose $\left\lfloor d_{G[X]}(v)/k\right\rfloor$ pairwise disjoint subsets
$$
S_1,\ldots,S_{\left\lfloor d_{G[X]}(v)/k\right\rfloor}\subseteq N_{G[X]}(v),
$$
each of size $k$.

Fix one such set $S_i$. 
For each $u\in S_i$, define
$$Y_u=N_G(v)\cap N_G(u)\cap Y.$$
We claim that the family $\{Y_u:u\in S_i\}$ has no system of distinct representatives.
Otherwise, choosing distinct vertices $w_u\in Y_u$ for all $u\in S_i$ would give $k$ pairwise disjoint edges $uw_u$ in $G[N_G(v)]$.
This would yield a copy of $F_k$ centered at $v$, a contradiction.

By Hall's theorem, there exists a nonempty set $I_i\subseteq S_i$ such that
\begin{equation}\label{eq:pf:friendship-graphs-stability-Hall}
  \left|\bigcup_{u\in I_i}Y_u\right|\le |I_i|-1\le |S_i|-1 = k-1.
\end{equation}
Since $w\notin Y_u$ implies $w\notin N_G(u)$ for every $w\in N_G(v)\cap Y$, every vertex
$$w\in \left(N_G(v)\cap Y\right) \setminus \bigcup_{u\in I_i}Y_u$$
is non-adjacent to every vertex of $I_i$.
Thus $I_i$ contributes at least
$$|N_G(v)\cap Y|-\left|\bigcup_{u\in I_i}Y_u\right|\ge d_{G[Y]}(v)-(k-1)$$
missing crossing edges, i.e., edges in $(X\times Y)\setminus E(G)$, by \eqref{eq:pf:friendship-graphs-stability-Hall}.
By Claim~\ref{claim:pf:friendship-graphs-stabilitymaxcut-degree}, since $v\in X$, we have $d_{G[Y]}(v)\ge d_{G[X]}(v)$. 
So
$$|N_G(v)\cap Y|-\left|\bigcup_{u\in I_i}Y_u\right|\ge d_{G[X]}(v)-(k-1).$$

Since the sets $I_i$ are pairwise disjoint for different indices $i$, the missing crossing edges counted for different indices $i$ are distinct. 
Thus
$$r\ge \left\lfloor \frac{d_{G[X]}(v)}{k}\right\rfloor\left(d_{G[X]}(v)-k+1\right).$$
Since $d_{G[X]}(v)\ge 2k$, we have
$$
\left\lfloor \frac{d_{G[X]}(v)}{k}\right\rfloor
\ge \frac{d_{G[X]}(v)}{2k}
\qquad\text{and}\qquad
d_{G[X]}(v)-k+1\ge \frac{d_{G[X]}(v)}{2}.$$
Therefore,
$$
r\ge 
\left\lfloor \frac{d_{G[X]}(v)}{k}\right\rfloor
\left(d_{G[X]}(v)-k+1\right)
\ge \frac{d_{G[X]}(v)^2}{4k},
$$
which implies
$$d_{G[X]}(v)\le 2\sqrt{kr}\le 4k\sqrt{r+1}.$$
\end{proof}

Next, we give bounds on $\nu(G[X])$ and $\nu(G[Y])$.  
\begin{claim}\label{claim:pf:friendship-graphs-stability-matching-number}
$\nu(G[X])\le k\left(1+3r/n\right)$ and $\nu(G[Y])\le k\left(1+3r/n\right)$.
\end{claim}
\begin{proof}
By symmetry, it suffices to prove that $\nu(G[X])\le k\left(1+3r/n\right)$.
Take a maximum matching in $G[X]$ and partition
$\left\lfloor \nu(G[X])/k\right\rfloor k$ of its edges into
$\left\lfloor \nu(G[X])/k\right\rfloor$ disjoint groups of $k$ matching edges.  

Consider one such group, say $x_1x'_1,\ldots,x_kx'_k$.
For each $j$, define
$$Y_j=N_G(x_j)\cap N_G(x'_j)\cap Y.$$
We must have $\bigcap_{j=1}^k Y_j=\varnothing$; otherwise, a vertex $y$ in this intersection would be adjacent to both endpoints of each of the $k$ disjoint internal edges $x_jx'_j$, and these edges together with $y$ would form a copy of $F_k$ centered at $y$.  
It follows that, for every $y\in Y$, at least one crossing edge from $y$ to $\{x_1,x'_1,\ldots,x_k,x'_k\}$ is missing.
Thus this group contributes at least $|Y|$ missing crossing edges.

Since different groups use disjoint vertices in $X$, the corresponding missing edges are distinct.
It means that
$$r\ge \left\lfloor \frac{\nu(G[X])}{k}\right\rfloor |Y|.$$
Since
$$\left\lfloor \frac{\nu(G[X])}{k}\right\rfloor\ge \frac{\nu(G[X])}{k}-1$$
and $|Y|\ge n/3$ by Claim~\ref{claim:pf:friendship-graphs-stability-XYsize}, we obtain
$$\frac{r}{|Y|}\ge \frac{\nu(G[X])}{k}-1,$$
and hence
$$\nu(G[X])\le k\left(1+\frac{r}{|Y|}\right)\le k\left(1+\frac{3r}{n}\right).$$
\end{proof}

\begin{claim}\label{claim:pf:friendship-graphs-stability-s-inequality}
$s\leq 16k^2 \sqrt{s+\alpha+1}+(s+\alpha+1)/4$.
\end{claim}
\begin{proof}
Let $M$ be a maximum matching in $G[X]$, and let $T$ be the set of endpoints of the edges in $M$.  
Clearly, $|T|=2\nu(G[X])$.  
The set $T$ covers all edges of $G[X]$; otherwise, an uncovered edge could be added to $M$, contradicting the fact that $M$ is maximum.
Then
\begin{equation*}\label{eq:pf:friendship-graphs-stability-eGX}
e(G[X])\le \sum_{t\in T}d_{G[X]}(t)\le 2\nu(G[X]) \Delta(G[X]).
\end{equation*}
Using \eqref{eq:r-q-simple} together with Claims~\ref{claim:pf:friendship-graphs-stability-max-degree} and~\ref{claim:pf:friendship-graphs-stability-matching-number}, we obtain
\begin{equation*}\label{eq:pf:friendship-graphs-stability-X-side}
e(G[X])\le 8k^2\left(1+\frac{3r}{n}\right)\sqrt{r+1}
\le 8k^2\left(1+\frac{3(s+\alpha)}{n}\right)\sqrt{s+\alpha+1}.
\end{equation*}
By symmetrical, we get 
\begin{equation*}\label{eq:pf:friendship-graphs-stability-Y-side}
e(G[Y])\le 8k^2\left(1+\frac{3(s+\alpha)}{n}\right)\sqrt{s+\alpha+1}.
\end{equation*}
Therefore
\begin{equation*}\label{eq:pf:friendship-graphs-stability-XY-side}
s=e(G[X])+e(G[Y])\le 16k^2\sqrt{s+\alpha+1} +\frac{48k^2}{n}(s+\alpha+1)^{3/2}.
\end{equation*}

It remains to show that
$$\frac{48k^2}{n}(s+\alpha+1)^{3/2}\le \frac{s+\alpha+1}{4}.$$
By \eqref{eq:pf:friendship-graphs-stability-n0} and Claim~\ref{claim:pf:friendship-graphs-stability-bound-of-rs}, we have
$\alpha+1\leq \varepsilon n^2$ and $s\leq \varepsilon n^2$, respectively.
We get  $s+\alpha+1\leq 2\varepsilon n^2$.
Thus
$$\frac{48k^2}{n}(s+\alpha+1)^{3/2}
  =48k^2(s+\alpha+1)\frac{\sqrt{s+\alpha+1}}{n}
  \le 48k^2\sqrt{2\varepsilon}\,(s+\alpha+1)
  \le \frac{s+\alpha+1}{4},$$
where the last inequality is due to \eqref{eq:pf:friendship-graphs-stability-eps} (that is, $48k^2\sqrt{2\varepsilon}\le 1/4$).
\end{proof}

Define $$\gamma=\max\{(64k^2)^2,5(\alpha+1)\}\quad \text{ and } \quad\beta=4\gamma+1.$$
If $s+\alpha+1\ge \gamma$, then
$$64k^2\leq \sqrt{s+\alpha+1} \qquad \text{and} \qquad\alpha+1\le \frac{s}{4}.$$
Combining this with Claim \ref{claim:pf:friendship-graphs-stability-s-inequality}, we conclude that
$$s\leq 16k^2 \sqrt{s+\alpha+1}+\frac{s+\alpha+1}{4}\le\frac{s+\alpha+1}{2}\le\frac{s+\frac{s}{4}}{2}=\frac{5s}{8},$$
a contradiction. 
So $s+\alpha+1< \gamma$.
Note that $q+r\leq s+\alpha$.
This implies that 
\begin{equation}\label{eq:s-r-q-final}
  s<\gamma<\beta,
  \qquad
  r<\gamma<\beta,
  \qquad \text{and} \qquad
  q<\gamma<\beta,
\end{equation}

Since $|X|+|Y|=n$, we have
$$|X||Y|
=\frac{(|X|+|Y|)^2-(|X|-|Y|)^2}{4}
=\frac{n^2-(|X|-|Y|)^2}{4}.$$
If $n$ is even, then
$$q=\left\lfloor \frac{n^2}{4}\right\rfloor-|X||Y|
=\frac{(|X|-|Y|)^2}{4}.$$
If $n$ is odd, then
$$q=\frac{(|X|-|Y|)^2-1}{4}.$$
In both cases, we have $(|X|-|Y|)^2\le 4q+1$.  
Using \eqref{eq:s-r-q-final} and $\beta=4\gamma+1$, we obtain
$$\bigl||X|-|Y|\bigr|\le \sqrt{4q+1}\le \sqrt{4\gamma+1}<\beta.$$

Therefore, we have
$$e(G[X])+e(G[Y])=s\le \beta,\qquad
|(X\times Y)\setminus E(G)|=r\le \beta,\qquad
\bigl||X|-|Y|\bigr|\le \beta.$$
We finish the proof of the Lemma.
\end{proof}

\begin{lemma}\label{lem:heavy-CH}
Fix integers $k\geq 1$ and $\beta\geq 0$.  
Let $G$ be an $F_k$-free graph, and suppose that every edge of $G[X]$ has at least $|Y|-\beta$ common neighbors in $Y$.  
If $|Y|$ is sufficiently large in terms of $k$ and $\beta$, then
$$\nu(G[X])\le k-1,\qquad\Delta(G[X])\le k-1.$$
The symmetric statement holds with $X$ and $Y$ interchanged.
\end{lemma}

\begin{proof}
Suppose first that $G[X]$ contains a matching $x_1x_1',\ldots,x_kx_k'$ of size $k$.  
For each $i$, the set of common neighbors of $x_i$ and $x_i'$ in $Y$ has size at least $|Y|-\beta$.  
Hence, if $|Y|>k\beta$, then
$$\left|\bigcap_{i=1}^k \bigl(N_G(x_i)\cap N_G(x_i')\cap Y\bigr)\right|\ge |Y|-k\beta>0.$$
Choose $y$ in this intersection, the $k$ triangles
$yx_1x_1',\ldots,yx_kx_k'$ form a copy of $F_k$ centered at $y$, a contradiction. 
This proves $\nu(G[X])\le k-1$.

Now suppose that some $x\in X$ has $k$ distinct neighbors $x_1,\ldots,x_k$ inside $X$.  
For each $i\in[k]$, the pair $x,x_i$ has at least $|Y|-\beta$ common neighbors in $Y$.  
If $|Y|$ is large enough, we can greedily choose distinct vertices $y_1,\ldots,y_k\in Y$ such that $y_i$ is adjacent to both $x$ and $x_i$.  
Then the $k$ triangles $xx_1y_1,\ldots,xx_ky_k$ form a copy of $F_k$ centered at $x$, a contradiction. 
Therefore, $\Delta(G[X])\le k-1$.
\end{proof}

\begin{corollary}\label{cor:internal-at-most-f}
Under the hypotheses of Lemma~\ref{lem:heavy-CH}, we have
$$e(G[X])\le f(k-1,k-1).$$
The symmetric bound holds for $G[Y]$.
\end{corollary}

\begin{proof}
By Lemma~\ref{lem:heavy-CH}, we have $\nu(G[X])\le k-1$ and $\Delta(G[X])\le k-1$.
Applying the Chv\'atal--Hanson theorem to $G[X]$, we obtain
$e(G[X])\le f(k-1,k-1).$
\end{proof}

\subsection{Proof of Theorem \ref{thm:main-reduce}}
The lower bound in Theorem~\ref{thm:main-reduce} follows from Lemma~\ref{lem:lower-bound-main1}.
Thus, it suffices to establish the upper bound.
The main idea is as follows.
We apply Lemma~\ref{lem:friendship-graphs-stability} to an extremal graph $G$ to obtain a partition $V(G)=X\cup Y$.
We then show that the subgraphs $G[X]$ and $G[Y]$ satisfy certain structural properties.
Moreover, we prove that $G=H_n(P,Q,R)$ for some $k$-admissible triple $(P,Q,R)$.
Finally, we use Lemma~\ref{lem:lower-bound-main1} and analyze the case of equality.

\medskip

Let $G$ be an $n$-vertex $F_k$-free graph attaining $g(k,t,n)$.  
By Lemma~\ref{lem:lower-bound-main1},
\begin{equation}\label{eq:pf:main-reduce-lower-for-extremizer}
t e(G)+N(K_3,G)=g(k,t,n)\ge t\left\lfloor \frac{n^2}{4}\right\rfloor+f(k-1,k-1)n+c_k^*(t).
\end{equation}
Since $G$ is $F_k$-free, it follows from Theorem~\ref{thm:ZhuChen2023} that
$$N(K_3,G)\le f(k-1,k-1)n+C(k,t),$$
where $C(k,t)$ is a constant depending only on $k$ and $t$. 
Together with \eqref{eq:pf:main-reduce-lower-for-extremizer}, this yields
\begin{equation}\label{eq:edge-near-floor}
e(G)\ge \left\lfloor \frac{n^2}{4}\right\rfloor+\frac{c_k^*(t)-C(k,t)}{t}
  =\left\lfloor \frac{n^2}{4}\right\rfloor-O(1).
\end{equation}
Applying Lemma~\ref{lem:friendship-graphs-stability}, we obtain a partition $V(G)=X\cup Y$ such that
$$\bigl||X|-|Y|\bigr|\le \beta,\qquad
  s:=e(G[X])+e(G[Y])\le \beta,\qquad
  r:=|(X\times Y)\setminus E(G)| \le \beta,$$
where $\beta$ is the constant given by Lemma~\ref{lem:friendship-graphs-stability}. 
In particular, $\bigl||X|-|Y|\bigr|\le \beta$ implies $|X|=n/2+O(1)$ and $|Y|=n/2+O(1)$.

Since $|(X\times Y)\setminus E(G)| \le\beta$, every edge in $G[X]$ has at least $|Y|-\beta$ common neighbors in $Y$, 
and every edge in $G[Y]$ has at least $|X|-\beta$ common neighbors in $X$.  
Lemma~\ref{lem:heavy-CH} and Corollary~\ref{cor:internal-at-most-f} imply
\begin{equation*}
\nu(G[X])\le k-1, \qquad \Delta(G[X])\le k-1, \qquad e(G[X])\le f(k-1,k-1),
\end{equation*}
and
\begin{equation*}
\nu(G[Y])\le k-1, \qquad \Delta(G[Y])\le k-1, \qquad e(G[Y])\le f(k-1,k-1).
\end{equation*}
We next show that the two edge inequalities above are in fact equalities.

For $xx'\in E(G[X])$, let
$$d_Y(xx')=|N_G(x)\cap N_G(x')\cap Y|,$$
and define $d_X(yy')$ similarly for $yy'\in E(G[Y])$.

\begin{claim}\label{claim:pf:main-reduce-edges-equal-f}
$e(G[X])=e(G[Y])=f(k-1,k-1)$.
\end{claim}

\begin{proof}
Since every triangle is either contained in $X$ or $Y$, or has exactly one internal edge, we get
\begin{equation*}\label{eq:triangle-cut-count}
\begin{aligned}
  N(K_3,G)={}&
  \sum_{xx'\in E(G[X])}d_Y(xx')
  +\sum_{yy'\in E(G[Y])}d_X(yy')\\
  &+N(K_3,G[X])+N(K_3,G[Y]).
\end{aligned}
\end{equation*}
Note that $e(G)=|X||Y|-r+s$.
Since $r\le \beta$ and $s\le \beta$, we have
$N(K_3,G[X])+N(K_3,G[Y])=O(1)$ and $s-r=O(1)$.
Moreover, $d_Y(xx')\le |Y|$ for every $xx'\in E(G[X])$ and $d_X(yy')\le |X|$ for every $yy'\in E(G[Y])$.
Then
\begin{equation*}\label{eq:upper-via-pq}
te(G)+N(K_3,G)\le t|X||Y|+e(G[X])|Y|+e(G[Y])|X|+O(1).
\end{equation*}
If, say, $e(G[X])\le f(k-1,k-1)-1$, then, using $|Y|=n/2+O(1)$ and $|X|+|Y|=n$, we deduce that
\begin{align*}
t e(G)+N(K_3,G)
&\le t|X||Y|+e(G[X])|Y|+e(G[Y])|X|+O(1)\\ 
&\le t\left\lfloor \frac{n^2}{4}\right\rfloor
 +(f(k-1,k-1)-1)|Y|+f(k-1,k-1)|X|+O(1)\\ 
&\le t\left\lfloor \frac{n^2}{4}\right\rfloor
 +f(k-1,k-1)n-\frac{n}{2}+O(1).
\end{align*}
This contradicts \eqref{eq:pf:main-reduce-lower-for-extremizer} for sufficiently large $n$.  
Thus $e(G[X])=f(k-1,k-1)$. 
The same argument shows that $e(G[Y])=f(k-1,k-1)$.
\end{proof}

Let $$A=\{x\in X, d_{G[X]}(x)>0\},\,\, P=G[A],\,\, B=\{y\in Y, d_{G[Y]}(y)>0\},\,\, \text{ and } \,\, Q=G[B].$$
Therefore, Claim \ref{claim:pf:main-reduce-edges-equal-f} gives $P,Q\in\mathcal{P}_k$.
It remains to prove that all missing crossing edges lie inside $A\times B$ and that $(P,Q,R)$ is admissible, where $R=G[A,B]$.

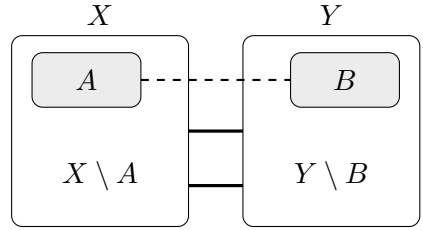
\begin{figure}[ht]
\centering
\begin{tikzpicture}[scale=0.9]
  \draw[rounded corners] (-3,-1.4) rectangle (-0.4,1.4);
  \draw[rounded corners] (0.4,-1.4) rectangle (3,1.4);
  \node at (-1.7,1.7) {$X$};
  \node at (1.7,1.7) {$Y$};
  \draw[rounded corners,fill=gray!15] (-2.7,0.35) rectangle (-1.1,1.15);
  \draw[rounded corners,fill=gray!15] (1.1,0.35) rectangle (2.7,1.15);
  \node at (-1.9,0.75) {$A$};
  \node at (1.9,0.75) {$B$};
  \node at (-1.7,-0.65) {$X\setminus A$};
  \node at (1.7,-0.65) {$Y\setminus B$};
  \draw[very thick] (-0.4,-0.8) -- (0.4,-0.8);
  \draw[very thick] (-0.4,0.0) -- (0.4,0.0);
  \draw[dashed,thick] (-1.1,0.75) -- (1.1,0.75);
\end{tikzpicture}
\caption{Construction of the graph $G$.}
\end{figure}

\begin{claim}\label{claim:pf:main-reduce-admissibility}
For every $a\in A$ and $b\in B$, we have
\begin{equation}\label{eq:pf:main-reduce-admissibility-a}
  d_P(a)+\nu\bigl(Q[N_G(a) \cap B]\bigr)\le k-1,
\end{equation}
and
\begin{equation*}\label{eq:pf:main-reduce-admissibility-b}
  d_Q(b)+\nu\bigl(P[N_G(b) \cap A]\bigr)\le k-1.
\end{equation*}
\end{claim}

\begin{proof}
It is enough to prove \eqref{eq:pf:main-reduce-admissibility-a}, because the other inequality is symmetric.  
Suppose it fails for some $a\in A$.  Let $M_Q(a)$ be a matching in $Q[N_G(a)\cap B]$ 
such that
$$d_P(a)+|M_Q(a)|\ge k.$$
For each $u\in N_P(a)$, the pair $a,u$ is an internal edge of $G[X]$, 
so by $|(X\times Y)\setminus E(G)| \le\beta$ it has at least $|Y|-\beta$ common neighbors in $Y$.  
Since $d_P(a)\le e(G[X])\le \beta$, $|B|$ is a constant, and $n$ is large, we may choose distinct vertices of $Y\setminus B$ adjacent to both $a$ and the different vertices $u\in N_P(a)$.  
The corresponding edges from $N_P(a)$ to these chosen vertices form a matching in $G[N_G(a)]$ of size $d_P(a)$.  
This matching is vertex-disjoint from $M_Q(a)$, since the chosen vertices lie in $Y\setminus B$, whereas all vertices of $M_Q(a)$ lie in $B$.  
Together this matching with the matching $M_Q(a)$, which is disjoint from $N_P(a)$ to these chosen vertices, this forms a matching of size at least $k$ inside $G[N_G(a)]$.  
This is a copy of $F_k$ centered at $a$, a contradiction.
\end{proof}

\begin{claim}\label{claim:pf:main-reduce-crossing-nonedge-outside}
Every missing crossing edge of $G$ lies in $A\times B$. Equivalently,
$d_G(x)=|Y|$ for every $x\in X\setminus A$ and $d_G(y)=|X|$ for every $y\in Y\setminus B$.
\end{claim}

\begin{proof}
Assume that $xy$ is a missing crossing edge, that is, $xy\in (X\times Y)\setminus E(G)$, where $x\in X$, $y\in Y$, and either $x\notin A$ or $y\notin B$.  
We show that $G+xy$ is still $F_k$-free.  
This contradicts the extremality of $G$, because adding an edge increases $t e(G)+N(K_3,G)$ by at least $t$, and $t\geq 1$.
Let $G'=G+xy$.  
We will prove that every vertex $v\in V(G')$ satisfies
$$\nu(G'[N_{G'}(v)])\leq k-1.$$

We first consider a vertex $v\in X\setminus A$. 
By the definition of $A$, we have $d_{G[X]}(v)=0$.
Thus $N_{G'}(v)\subseteq Y$.
Since the added edge $xy$ is a crossing edge, it does not change the edges inside $Y$; hence $G'[Y]$ is obtained from $Q$ by adding isolated vertices.
Thus,
$$\nu(G'[N_{G'}(v)])\le \nu(G'[Y])=\nu(Q)\le k-1.$$
The same argument applies to every $v\in Y\setminus B$.

Now consider a vertex $a\in A$. 
Every edge in $G'[N_{G'}(a)]$ is either incident with a vertex of $N_P(a)$ or is an edge of $Q[N_{G'}(a)\cap B]$.  
Therefore, every matching in $G'[N_{G'}(a)]$ has size at most
\[
d_P(a)+\nu(Q[N_{G'}(a)\cap B]).
\]
Since $xy\notin A\times B$, the added edge does not change $N_G(a)\cap B$. 
Namely, 
\[
N_{G'}(a)\cap B=N_G(a)\cap B.
\]
By Claim~\ref{claim:pf:main-reduce-admissibility}, this upper bound is at most $k-1$.  
Thus $a$ cannot be the center of a copy of $F_k$ in $G'$.  
The proof for a vertex $b\in B$ is symmetric, again using Claim~\ref{claim:pf:main-reduce-admissibility}.

All vertices have now been checked, so $G'$ is $F_k$-free.  
This contradicts the extremality of $G$, and the claim holds.
\end{proof}

By Claim~\ref{claim:pf:main-reduce-crossing-nonedge-outside}, the set of crossing edges has the form
$$E_G(X,Y)=\bigl((X\times Y)\setminus(A\times B)\bigr)\cup E(R).$$
Claim~\ref{claim:pf:main-reduce-admissibility} now becomes exactly the admissibility conditions \eqref{eq:adm-A} and \eqref{eq:adm-B}, 
as $N_G(a)\cap B=N_R(a)$ and $N_G(b)\cap A=N_R(b)$. 
Thus $(P,Q,R)$ is $k$-admissible.

Since $|X||Y|\le \left\lfloor n^2/4\right\rfloor$ and $\Phi(P,Q,R,t)\le c_k^*(t)$, using the same counting as in Lemma~\ref{lem:lower-bound-main1}, we have
\begin{align}
t e(G)+N(K_3,G)
&=t|X||Y|+f(k-1,k-1)n+\Phi(P,Q,R,t)\nonumber\\
&\leq t\left\lfloor \frac{n^2}{4}\right\rfloor+f(k-1,k-1)n+c_k^*(t).
\label{eq:exact-count-unbalanced}
\end{align}
Combining this with the lower bound \eqref{eq:pf:main-reduce-lower-for-extremizer}, we obtain \eqref{eq:exact-M}.

Finally, if $G$ is extremal, then equality must hold in \eqref{eq:exact-count-unbalanced}.  
So
$$|X||Y|=\left\lfloor{n^2/4}\right\rfloor\qquad \text{and} \qquad \Phi(P,Q,R,t)=c_k^*(t).$$
This proves the extremal structure and completes the proof of Theorem~\ref{thm:main-reduce}.
\qed

\section{Concluding remarks}\label{concluding-remarks}
By a small modification of the proof, Theorem~\ref{thm:main} can be extended to vertex-disjoint unions of different friendship graphs as follows.

\begin{theorem}
Let $t\geq 1$ and $\ell_1\geq \ell_2\geq \cdots\geq \ell_{t+1}\geq 3$ be integers.
For sufficiently large $n$, we have
\begin{equation*}
  \mathrm{ex}\left(n,K_3,\bigcup_{i=1}^{t+1}F_{\ell_i}\right)
  =\binom{t}{3}+(n-t)\binom{t}{2}
  +t\left\lfloor \frac{(n-t)^2}{4} \right\rfloor
  +f(\ell_{t+1}-1,\ell_{t+1}-1)(n-t)
  +c_{\ell_{t+1}}^*(t).
\end{equation*}
Moreover, every extremal graph is of the form $K_t \vee H_{n-t}(P,Q,R)$ for some $\ell_{t+1}$-admissible triple $(P,Q,R)$ with $\Phi(P,Q,R,t)=c_{\ell_{t+1}}^*(t)$.
\end{theorem}

Recall that 
\begin{eqnarray*}
  \Phi(P,Q,R,t)&=&( 2f(k-1,k-1)-|A||B|+e(R))t-f(k-1,k-1)(|A|+|B|)\nonumber\\
  &&\quad +N(K_3,P)+N(K_3,Q)+\tau_R(P,Q).
\end{eqnarray*}
and
\begin{equation*}
  c_k^*(t)=\max\{\Phi(P,Q,R,t):(P,Q,R,t)\text{ is }k\text{-admissible}\}
\end{equation*}

The finite optimization explains precisely why adding suitable edges between $A$ and $B$ can improve the constant.  
Compared with deleting all edges of $A\times B$, a graph $R\subseteq A\times B$ changes the constant by
$$e(R)t+\tau_R(P,Q).$$
If $R$ is a matching, then $\tau_R(P,Q)=0$, 
so every added matching edge improves the value by exactly $t$.

For odd $k$, a natural Chv\'atal--Hanson construct is $2K_k$.  
We have the following proposition.
We write $P=P_1\cup P_2$ and $Q=Q_1\cup Q_2$, where each $P_i$ and $Q_j$ is a copy of $K_k$.

\begin{proposition}\label{prop:odd-crossing}
Let $k\ge 3$ be odd and $P=Q=2K_k$.  
If $(P,Q,R)$ is $k$-admissible, then for every $i,j\in\{1,2\}$ the bipartite graph $R[P_i,Q_j]$ is a matching.
In particular, $$e(R)\le 4k.$$
The bound is attained by taking a perfect matching between each pair $P_i,Q_j$.
\end{proposition}

\begin{proof}
For every $a\in V(P)$ we have $d_P(a)=k-1$.  
The admissibility condition \eqref{eq:adm-A} implies $$\nu(Q[N_R(a)])=0.$$
Thus $N_R(a)$ is an independent set in $Q$, so it contains at most one vertex from each clique $Q_j$.  
Thus each vertex of $P_i$ has degree at most one into each $Q_j$.  
By the symmetric admissibility condition, each vertex of $Q_j$ has degree at most one into each $P_i$. 
It follows that $R[P_i,Q_j]$ is a matching and has at most $k$ edges. 
Summing over the four pairs $P_i,Q_j$ yields that $e(R)\le 4k$.
Moreover, $e(R)=4k$ implies each pair $P_i,Q_j$ is a perfect matching.
\end{proof}

It is worth noting that $\tau_R(2K_k,2K_k)=0$ and
$$\Phi(2K_k,2K_k,R,t)=(e(R)-2k^2-2k)t-\frac{k(k-1)(10k+4)}{3}.$$
If $R$ is a perfect matching between each pair $P_i,Q_j$, then
$$\Phi(2K_k,2K_k,R,t)=-\frac{10k+6t+4}{3}k(k-1)$$
and hence
$$c_k^*(t)\geq -\frac{10k+6t+4}{3}k(k-1)\quad  \text{ for $k$ is odd}.$$

An interesting problem is to determine the exact value of $c^*_k(t)$.

\begin{problem}
For integers $k\geq 1$ and $t\geq 1$, determined the exact value of $c^*_k(t)$.
\end{problem}

%Therefore, we pose the following conjecture.

%\begin{conjecture}
%For $k$ is odd, we have $$c_k^*(t)= -\frac{10k+6t+4}{3}k(k-1).$$
%\end{conjecture}

\section*{Acknowledgement}
\noindent This research is supported by National Key R\&D Program of China under grant number 2024YFA1013900, NSFC under grant number 12471327,  JSPS KAKENHI Grant Number 25KF0036, the NSF of Hubei Province Grant Number 2025AFB309,  the China Postdoctoral Science Foundation  Grant Number 2025M773113, the Fundamental Research Funds for the Central Universities, Central China Normal University Grant Number CCNU24XJ026.

\section*{Declaration}
	
\noindent$\textbf{Conflict~of~interest}$
The authors declare that they have no known competing financial interests or personal relationships that could have appeared to influence the work reported in this paper.
	
\noindent$\textbf{Data~availability}$
No data was used for the research described in the article.

\end{document}